\documentclass[11pt,reqno]{amsart}
\usepackage{latexsym,amsmath}
\usepackage[left=3cm,top=2.5cm,right=3cm,bottom=2.5cm]{geometry}
\usepackage{amsthm}
\usepackage{amssymb}
\usepackage{epsfig}
\usepackage{hyperref}
\usepackage{longtable}
\usepackage{verbatim}
\usepackage{graphicx} 
\usepackage{graphics}
\usepackage{geometry}
\usepackage{tabularx}
\usepackage[toc,page]{appendix}

\usepackage{pgfplots} % fazer plots
\usepackage{float}

\DeclareMathOperator{\ex}{ex}

\def\HH{\mathcal{H}}

\def\C{\mathcal{C}}

\usepackage[normalem]{ulem}

\begin{document}
\allowdisplaybreaks[2]

\newtheorem{theorem}{Theorem}[section]
\newtheorem{cor}[theorem]{Corollary}
\newtheorem{lemma}[theorem]{Lemma}
\newtheorem{fact}[theorem]{Fact}
\newtheorem{prob}[theorem]{Problem}
\newtheorem{property}[theorem]{Property}
\newtheorem{corollary}[theorem]{Corollary}
\newtheorem{proposition}[theorem]{Proposition}
\newtheorem{claim}[theorem]{Claim}
\newtheorem{conjecture}[theorem]{Conjecture}
\newtheorem{definition}[theorem]{Definition}
\theoremstyle{definition}
\newtheorem{example}[theorem]{Example}
\newtheorem{remark}[theorem]{Remark}
\newcommand\eps{\varepsilon}

\title{Edge-colorings avoiding patterns in a triangle}

\author[C. Hoppen]{Carlos Hoppen}
\address{Instituto de Matem\'atica e Estat\'{i}stica, UFRGS -- Avenida Bento Gon\c{c}alves, 9500, 91501--970 Porto Alegre, RS, Brazil}
\email{choppen@ufrgs.br}

\author[H. Lefmann]{Hanno Lefmann}
\address{Fakult\"at f\"ur Informatik, Technische Universit\"at Chemnitz,
Stra\ss{}e der Nationen 62, 09111 Chemnitz, Germany}
\email{Lefmann@Informatik.TU-Chemnitz.de}

\author[D. R. Schmidt ]{Dionatan Ricardo Schmidt}
\address{Instituto de Matem\'atica e Estat\'{i}stica, UFRGS -- Avenida Bento Gon\c{c}alves, 9500, 91501--970 Porto Alegre, RS, Brazil}
\email{dionatan.schmidt@ufrgs.br}

\thanks{This work was partially supported by CAPES and DAAD via Probral (CAPES Proc.~88881.143993/2017-01 and DAAD~57391132 and~57598268). The first author acknowledges the support of CNPq~315132/2021-3), Conselho Nacional de Desenvolvimento Cient\'{i}fico e Tecnol\'{o}gico. The third author acknowledges the support by CAPES}

\begin{abstract}
For positive integers $n$ and $r$, we consider $n$-vertex graphs with the maximum number of $r$-edge-colorings with no copy of a triangle where exactly two colors appear. We prove that, if $2 \leq r \leq 26$ and $n$ is sufficiently large, the maximum is attained by the bipartite Tur\'{a}n graph $T_2(n)$ on $n$ vertices. This is best possible, as $T_2(n)$ is not extremal for $r \geq 27$ colors and $n \geq 3$. 
\end{abstract}

\maketitle

\section{Introduction}
Starting with a question of Erd\H{o}s and Rothschild~\cite{erdos}, there has been substantial interest in the problem of characterizing graphs that admit the largest number of $r$-edge-colorings avoiding a fixed pattern of a graph $F$, where the number $r$ of colors and the pattern of $F$ are given. To be precise, fix an integer $r\geq 2$ and a graph $F$. We say that $P$ is an \emph{$r$-pattern} of $F$ if it is a partition of the edge set of $F$ into at most $r$ classes. An $r$-edge-coloring (or $r$-coloring, for short) of a graph $G$ is said to be \emph{$P$-free} if $G$ does not contain a copy of $F$ such that the partition of the edge set induced by the coloring is isomorphic to $P$. We write $c_{r,P}(G)$ for the number of $P$-free $r$-colorings of a graph $G$ and we define 
$$c_{r,P}(n)=\max\{c_{r,P}(G)\colon |V(G)|=n\}.$$
An $n$-vertex graph $G$ such that $c_{r,P}(G)=c_{r,P}(n)$ is said to be \emph{$(r,P)$-extremal}. 

We focus on the case when $F$ is the triangle $K_3$. There are three possible patterns: the monochromatic pattern $K_3^M$, the rainbow pattern $K_3^R$, and the pattern $K_3^{(2)}$ with two classes, one containing two edges and one containing a single edge, depicted in the figure below.

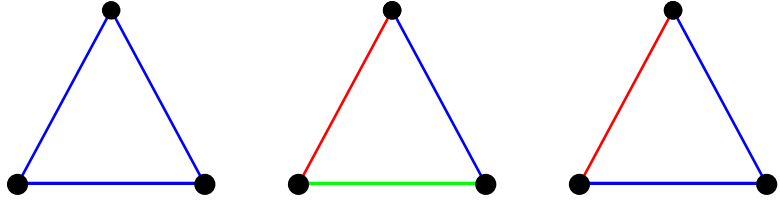
\begin{figure}[H]

%	\begin{center}
%		\begin{tabular}{cc} 
			\begin{tikzpicture}[line cap=round,line join=round]
	\begin{axis}[
width=0.9\textwidth,
height=0.25\textwidth,
xmin=-1, xmax=9,
ymin=0, ymax=0.5,
minor tick num=3,
grid=none,
axis lines=none]

\addplot[mark=*,mark size=3,mark options={draw=black,fill=black},line width=1pt,blue,samples=1] coordinates {(2,0)(1,0.5)(0,0)};

\addplot[mark=*,mark size=3,mark options={draw=black,fill=black},line width=2pt,blue,samples=1] coordinates {(0,0)(2,0)};

\addplot[mark=*,mark size=3,mark options={draw=black,fill=black},line width=1pt,blue,samples=1] coordinates {(5,0)(4,0.5)};

\addplot[mark=*,mark size=3,mark options={draw=black,fill=black},line width=2pt,green,samples=1] coordinates {(3,0)(5,0)};

\addplot[mark=*,mark size=3,mark options={draw=black,fill=black},line width=1pt,red,samples=1] coordinates {(3,0)(4,0.5)};

\addplot[mark=*,mark size=3,mark options={draw=black,fill=black},line width=2pt,blue,samples=1] coordinates {(6,0)(8,0)};

\addplot[mark=*,mark size=3,mark options={draw=black,fill=black},line width=1pt,blue,samples=1] coordinates {(8,0)(7,0.5)};

\addplot[mark=*,mark size=3,mark options={draw=black,fill=black},line width=1pt,red,samples=1] coordinates {(6,0)(7,0.5)};

\end{axis}
\end{tikzpicture} \\
	%	\end{tabular} 
	%\end{center}
	
	\label{pmpa}
	\caption{The patterns $K_3^M$, $K_3^R$, and $K_3^{(2)}$, respectively.}
\end{figure}

Regarding the monochromatic pattern $K_3^M$, and assuming that $n$ is sufficiently large, it is known that the balanced, complete bipartite Tur\'an graph $T_2(n)$ is the unique $(r,K_3^M)$-extremal graph if $r\in \{2,3\}$~\cite{alon,yuster}, and that $T_2(n)$ is not $(r,K_3^M)$-extremal for $r\geq 4$. Extremal configurations are known for $r=4$~\cite{yilma}, for $r\in \{5,6\}$~\cite{botler} and for $r=7$~\cite{PS2022}. We should note that, for $r=5$, this problem admits several non-isomorphic configurations that achieve the extremal value in an asymptotic sense. The exact extremal graphs are not known for all sufficiently large $n$. Moreover, even though the extremal configurations are not known for larger values of $r$, structural properties of these configurations have been studied in~\cite{PS2021,PS2022}. 

Regarding the rainbow pattern $K_3^R$, again assuming that $n$ is sufficiently large, it is known that the complete graph $K_n$ is $(r,K_3^M)$-extremal for $r\in \{2,3\}$~\cite{baloghli}, and that the Tur\'{a}n graph $T_2(n)$ is $(r,K_3^M)$-extremal for $r \geq 4$~\cite{baloghli,rainbow_triangle}. We also refer the reader to~\cite{BBH20} for related results.

The first results about the pattern $K_3^{(2)}$ have been obtained by two of the current authors~\cite{2coloredlagos}, who showed that $T_2(n)$ is $(r,K_3^{(2)})$-extremal for $2 \leq r \leq 12$ and sufficiently large $n$. They also observed that this conclusion cannot be extended to $r \geq 27$, as the following example illustrates. Let $n$ be a positive integer (for the sake of the argument, assume that it is divisible by $4$), and consider the complete, balanced $4$-partite Tur\'an graph $T_4(n)$ with vertex partition $V_1\cup\cdots\cup V_4$, where each class has size $n/4$. Let $C$ be a set of 27 colors and partition $C$ as $C_1\cup C_2\cup C_3$, where $|C_i|=9$ for $1\leq i\leq 3$. Consider the colorings of $T_4(n)$ that assign colors in $C_1$ to edges between $V_1$ and $V_2$ and to edges between $V_3$ and $V_4$; colors in $C_2$ to edges between $V_1$ and $V_3$ and to edges between $V_2$ and $V_4$; colors in $C_3$ to edges between $V_1$ and $V_4$ and to edges between $V_2$ and $V_3$. Note that no copy of $K_3^{(2)}$ may appear in such a coloring (indeed, all triangles are rainbow), and that the number of such colorings is equal
$$9^{\frac{6n^2}{16}}=27^{\frac{n^2}{4}}=c_{27,K_3^{(2)}}(T_2(n)).$$
Moreover, other colorings of $T_4(n)$ may be produced, for instance by choosing a different partition $C_1\cup C_2\cup C_3$. This construction shows that $T_2(n)$ is not $(r,K_3^{(2)})$-extremal for any $r\geq 27$.

In this paper, we prove that the result in~\cite{2coloredlagos} may be extended all the way to $r=26$, as we now state.
\begin{theorem}\label{theorema:main}
Given a number $r$ of colors satisfying $2 \leq r \leq 26$, there exists $n_0$ such that, for every $n \geq n_0$ and every $n$-vertex graph $G$, we have
\begin{equation}\label{eq_main}
c_{r,K_3^{(2)}}(G) \leq r^{\ex(n,K_3)}.
\end{equation}
Moreover, equality holds in~\eqref{eq_main} for $n \geq n_0$ if and only if $G$ is isomorphic to the bipartite Tur\'{a}n graph $T_2(n)$.
\end{theorem}

The work in \cite[Lemma~4.4]{eurocomb15} implies that to prove Theorem~\ref{theorema:main}, it suffices to show a related stability result establishing that any $n$-vertex graph with a `large' number of $r$-colorings must be `almost bipartite'. For a formal statement of this stability result, given a graph $G=(V,E)$ and a set $W \subseteq V$, we write $e_G(W)$ for the number $|E(G[W])|$ of edges in the subgraph of $G$ induced by $W$. 
\begin{lemma} \label{lemma:main_result}
Let $2 \leq r \leq 26$.  For any fixed $\delta > 0$, there exists $n_0$ such that the following holds for all $n\geq n_0$. If $G=(V,E)$ is an $n$-vertex graph such that
$$ c_{r,K_3^{(2)}}(G) \geq r^{\ex(n, K_3)},$$  
then there is a partition $V = W_1 \cup  W_2$ of its vertex set such that $e_G(W_1) + e_G(W_2) \leq \delta n^2$.
\end{lemma} 

The proof of our results combines the regularity method used in \cite{alon} and~\cite{eurocomb15} with linear programming. In some previous applications of this method, the general bounds provided by linear programming were not strong enough to extend the conclusion of Lemma~\ref{lemma:main_result} to the entire range of $r$. In this paper we use an inductive component in the proof, which allows us to better exploit local restrictions and to extend the result in~\cite{2coloredlagos} to all values of $r$ for which it was conjectured to hold. The remainder of the paper is organized as follows. In the next section, we introduce the basic preliminary results and notation required. In Section~\ref{sec:main_proof}, we prove Lemma~\ref{lemma:main_result}.

\section{Notation and auxiliary tools}

In this section, we introduce some notation and auxiliary tools that will be useful for our purposes.

\subsection{Stability for Erd\H{o}s-Rothschild type problems} As mentioned in the introduction, results in~\cite{eurocomb15} ensure that if we prove the stability result stated in Lemma~\ref{lemma:main_result}, we immediately obtain our main result Theorem~\ref{theorema:main}. To describe why this is the case, we start with a definition.

\begin{definition}\label{colored_stability}
Let $F$ be a graph with chromatic number $\chi(F)=k \geq 3$ and let $P$ be a pattern of $F$. We say that the pair $(F,P)$ satisfies the Color Stability Property for a positive integer $r$ if, for every $\delta>0$, there exists $n_0$ with the following property. If $n>n_0$ and $G$ is an $n$-vertex graph such that $c_{r,P}(G) \geq r^{\ex(n,F)}$, then there exists a partition $V(G)=V_1 \cup \cdots \cup V_{k-1}$ such that $\sum_{i=1}^{k-1} e(V_i) < \delta n^2$.
\end{definition}
The authors of~\cite{eurocomb15} have proved that, under the technical conditions below, if we show that a pair $(F,P)$ satisfies the Color Stability Property for a positive integer $r$, then we may immediately conclude that the Tur\'{a}n graph $T_{k-1}(n)$ is the unique $(r,P)$-extremal graph for sufficiently large $n$. In the next statement, a pattern $P$ of $K_{k}$ is \emph{locally rainbow} if there is a vertex that is incident with edges in $k-1$ different classes of $P$. For $K_3$, the patterns $K_3^R$ and $K_3^{(2)}$ are locally rainbow patterns, but $K_3^M$ is not.
\begin{theorem}\cite[Lemma 4.4]{eurocomb15} \label{exact} 
Let  $k \geq 3$ and let $P$ be a locally rainbow pattern of $K_{k}$ such that $(K_{k},P)$ satisfies the Color Stability Property of Definition~\ref{colored_stability} for a positive integer 
\begin{equation*}
r \geq
\begin{cases}
3 & \textrm{ if $k=3$}\\
\lceil ek \rceil & \textrm{ if $k\geq 4$},
\end{cases}
\end{equation*}
where $e$ denotes Euler's number. Then there is $n_0$ such that every graph of order $n > n_0^2$ has at most $r^{\ex(n,K_{k})}$ distinct $(K_{k},P)$-free $r$-edge colorings. Moreover, the only graph on $n$ vertices for which the number of such colorings is $r^{\ex(n,K_{k})}$ is the Tur\'{a}n graph $T_{k}(n)$.
\end{theorem}
This justifies why we only need to prove Lemma~\ref{lemma:main_result} to derive our exact result.  

\subsection{Regularity and embeddings} Our strategy relies on a colored version of the celebrated Szemer\'{e}di Regularity Lemma. To state it, we need some terminology. Let $G = (V,E)$ be a graph, and let $A$ and $B$ be two subsets of $V(G)$. If $A$ and $B$ are non-empty, define the edge density between $A$ and $B$ by
$$d_G(A,B) = \frac{e_G(A,B)}{|A||B|},$$
where $e_G(A,B)$ is the number of edges with one vertex in $A$ and the other in $B$. %(When $A=B$, we write $e_G(A,A)=e_G(A)$. 
For $\eps > 0$ the pair $(A,B)$ is called \emph{$\eps$-regular} if, for all subsets $X \subseteq A$ and $Y \subseteq B$ satisfying $|X| \geq \eps|A|$ and $|Y| \geq \eps|B|$, we have
$$|d_G(X,Y) - d_G(A,B)| < \eps.$$
An \emph{equitable partition} of a set $V$ is a partition of $V$ into pairwise disjoint classes $V_1,\ldots,V_m$ satisfying $\arrowvert |V_i| - |V_j| \arrowvert \leq 1$ for all pairs $i,j$. An equitable partition of the vertex set $V$ of $G$ into classes $V_1,\ldots,V_m$ is called \emph{$\eps$-regular} if at most $\eps \binom{m}{2}$ of the pairs $(V_i, V_j)$ are not $\eps$-regular.

We are now ready to state a colored version of the Regularity Lemma, whose proof may be found in \cite{kosi}. For a positive integer $r$, we use the standard notation $[r] = \{1, \ldots , r\}$.
\begin{lemma} \label{lemma:regularity}
For every $\eps > 0$ and every positive integer $r$, there exists $M = M(\eps,r)$ such that the following holds. If the edges of any graph $G$ of order $n > M$ are $r$-colored $E(G) = E_1 \cup \cdots \cup E_r$, then there is a partition of the vertex set $V(G) = V_1 \cup \cdots \cup V_m$, with $1/\eps \leq m \leq M$, which is $\eps$-regular simultaneously with respect to the graphs $G_i = (V,E_i)$ for all $i \in [r]$.
\end{lemma}
A partition $V_1 \cup \cdots \cup V_m$ of $V(G)$ as in Lemma \ref{lemma:regularity} will be called a \textit{multicolored $\eps$-regular partition}. For $\eta>0$, we may define a \textit{multicolored cluster graph} $\HH(\eta)$ associated with this partition, where the vertex set is $[m]$ and $e = \{i,j\}$ is an edge of $\HH(\eta)$ if $\{V_i,V_j\}$ is a regular pair in $G$ \textit{for every} color $c \in [r]$ and the edge density of the partition is at least $\eta$ for some color $c \in [r]$. Each edge $e$ in $\HH(\eta)$ is assigned the list $L_{e}$ containing all colors for which its edge density is at least $\eta$, so that $|L_e| \geq 1$ for every edge in the multicolored cluster graph $\HH(\eta)$. Given an (edge)-coloring $\widehat{F}$ of a graph $F$, we say that a multicolored cluster graph $\HH$ contains $\widehat{F} $ if $\HH$ contains a copy of $F$ for which the color of each edge of $\widehat{F} $ is contained in the list of the corresponding edge in $\HH$.  More generally, if $F$ is a graph with color pattern $P$, we say that $\HH$ contains $(F, P)$ if it contains some coloring $\widehat{F}$ of $F$ with pattern $P$.

In connection with this definition, we shall use the following standard embedding result, which is a special case of~\cite[Lemma~2.4]{eurocomb15}.
\begin{lemma} \label{lemma:colored_subgraph} 
For every $\eta > 0$ and all positive integers $k$ and $r$, there exist $\eps = \eps (r, \eta, k) > 0$ and a positive integer $n_0(r, \eta, k)$ with the following property. Suppose that $G$ is an $r$-colored graph on $n > n_0$ vertices with a multicolored $\eps$-regular partition $V = V_1 \cup \cdots \cup V_m$ which defines the multicolored cluster graph $\HH = \HH(\eta)$. Let $F$ be a fixed $k$-vertex graph with a prescribed color pattern $P$ on $t \leq r$ classes. If $\HH$ contains $(F, P)$, then the graph $G$ also contains $(F, P)$.
\end{lemma}

\subsection{Stability}
Another basic tool in our paper is a stability result for graphs.
We shall use the following theorem by F\"uredi \cite{fu15}.
\begin{theorem} \label{theorem:stability_furedi} 
Let $G = (V,E)$ be a $K_{k}$-free graph on $n$ vertices. If $|E| = \ex(n, K_{k}) - t$, then there exists a partition  $V= V_1 \cup \cdots \cup  V_{k-1}$ with $\sum_{i = 1}^{k-1} e(V_i) \leq  t$. 
 \end{theorem}
 
We also use the following version of a simple lemma due to Alon and Yuster \cite{AY}. For completeness, we include its proof, that relies on the well known fact that any graph on $m\geq 1$ edges contains a bipartite subgraph with more than $m/2$ edges.
\begin{lemma} \label{lemma:AY} 
Let $0 < t < n^2/16$ and let $G$ be a $K_3$-free graph with $n$ vertices and with $\ex(n,K_{3})-t$ edges. If we produce a new graph $G'$ by adding at least $5t$ new edges to the graph $G$, then $G'$ contains a copy of $K_{3}$ with exactly one new edge.
\end{lemma}
\begin{proof}
Let $G=(V,E)$ be a $K_3$-free graph with $n$ vertices and with $\ex(n,K_{3})-t$ edges and let $F$ be a set of at least $5t$ new edges. Let $E' \subset E$ be a set with maximum cardinality such that $G'=(V,E')$ is bipartite. Let $V=V_1 \cup V_2$ be the bipartition of $G'$ and let $E''=E \setminus E'$. By Theorem~\ref{theorem:stability_furedi}, we have $|E''| \leq t$. 

The number of edges of $F$ with one end in $V_1$ and the other  in $V_2$ is at most 
$$|V_1||V_2|-|E'| \leq \ex(n,K_3)-|E'| \leq |E|+t-|E'| = t+|E''|.$$
So, if $F'$ denotes the subset of $F$ containing all edges with both ends in $V_1$ or with both ends in $V_2$, we have
$$|F'| \geq 5t - (t+|E''|)=4t-|E''|.$$ 
To conclude our proof, consider a maximum subset $F'' \subset F'\cup E''$ such that $G''=(V,F'')$ is bipartite. It is well known that 
$$|F''|>\frac{|F'\cup E''|}{2}\geq \frac{(4t-|E''|)+|E''|}{2}=2t.$$  

To conclude the proof, consider the graph $G^\ast=(V,E' \cup F'')$. Observe that the edges of $G^\ast$ with one end in $V_1$ and the other in $V_2$ are in $E'$, and the others are in $F''$. By the above considerations,  
$$|E'|+|F''| > \ex(n,K_{3})-(t+|E''|)+2t \geq \ex(n,K_{3}),$$ 
so that $G^\ast$ contains a triangle by Tur\'{a}n's Theorem. We claim that this triangle contains exactly one edge in $F''$. It is obvious that it must contain at least one edge in $F''$ and that it cannot contain an edge with both ends in $V_1$ and another with both ends in $V_2$.  Moreover, if the triangle contained two edges with all ends in one of the sides of the bipartition, say in $V_1$, the third edge of the triangle would also connect vertices in $V_1$, so that the three edges would lie in $F''$, contradicting the fact that $G''=(V,F'')$ is bipartite.
\end{proof}

%The following standard result will also be useful in our proofs.
%\begin{lemma} \label{lemma:maxcut}
%Let $\ell, m \geq 1$ be integers. If $G$ is a graph with $m$ edges, then $G$ contains an $\ell$-partite subgraph with more than $(\ell - 1)m/\ell$ edges.
%\end{lemma}

\section{Proof of Lemma~\ref{lemma:main_result}}\label{sec:main_proof}

In this section, we will prove Lemma~\ref{lemma:main_result}. To this end, fix $r \in \{2,\ldots,26\}$ and let $\delta>0$. With foresight, we consider auxiliary constants $\alpha$, $\xi$ and $\eta > 0$ such that
\begin{center}
\begin{equation}\label{eq:quantification}
\alpha=\frac{1}{1000}, \ \   \xi < \dfrac{\delta}{22}, \ \ \xi > 10^4\cdot H((r+1)\eta) + (10^4+1)\cdot  (r+1)\cdot \eta \ \ \text{ and } \ \ \eta < \dfrac{\delta}{2r},
\end{equation}
\end{center}
where $H \colon [0,1] \to [0,1]$ is the \emph{entropy function} given by $H(0) = H(1) = 0$ and by $H(x) = -x \log_2 x - (1-x) \log_2(1-x)$ for $x \in (0,1)$. 

Let $\varepsilon = \varepsilon(r,\eta, 3) > 0$ and $n_0 = n_0(r,\eta,3)$ satisfy the assumptions of Lemma~\ref{lemma:colored_subgraph}, and assume without loss of generality that $\varepsilon < \min\{\eta/2,1/n_0\}$. Fix $M = M(r,\varepsilon)$ given by Lemma~\ref{lemma:regularity}.

Given an $n$-vertex graph $G$ such that $n \geq n_0$, let $\C=\C(G)$ the set of all $K_3^{(2)}$-free $r$-colorings of $G$. By Lemma \ref{lemma:regularity} and the discussion following it, each coloring $\Phi \in \C$ is associated with a multicolored $\eps$-regular partition $V= V_1 \cup \cdots \cup V_m$, where $1 / \eps \leq m \leq M$. This partition is in turn is associated with a multicolored cluster graph $\HH=\HH(\eta)$. Our choice of parameters implies that $\HH$ must be $K_3^{(2)}$-free, otherwise the coloring of $G$ leading to it would contain a copy of $(K_3,\leq 2)$ by Lemma~\ref{lemma:colored_subgraph}.

Towards an upper bound on the size of $\C$, we determine an upper bound on the number of colorings that give rise to a fixed partition $ V_1 \cup \cdots \cup V_m$ and to a fixed multicolored cluster graph $\HH$. We first consider the edges of $G$ whose colors are not captured by the lists $L_e$ associated with edges $e\in E(\HH)$. Lemma~\ref{lemma:regularity} ensures that, for each color in $[r]$, there are at most $\eps \binom{m}{2}$ irregular pairs with respect to the partition $V = V_1 \cup \dots \cup V_m$, hence at most 
\begin{align}
r \cdot \eps \cdot \binom{m}{2} \cdot \left(\frac{n}{m}\right) ^ 2 \leq r \eps \cdot n ^ 2 \leq \frac{r \eta}{2} \cdot n ^ 2 \label{eq:irregular}
\end{align}
edges of $G$ are contained in an irregular pair with respect to one of the colors. Moreover, there are at most
\begin{align}
m \cdot \left(\frac{n}{m}\right) ^ 2 = \frac{n ^ 2}{m} \leq \eps n ^ 2 \leq \frac{\eta}{2} \cdot n ^ 2 \label{eq:inside}
\end{align}
edges with both ends in the same class $V_i$. Finally, we consider edges $f$ whose endpoints are in distinct classes $V_i$ and $V_j$ and such that the edge density of the edges with the color of $f$ is less than $\eta$ with respect to this pair. The number of edges of this type is at most
\begin{align}
r \cdot \eta \cdot \binom{m}{2} \cdot \left(\frac{n}{m}\right) ^ 2 \leq \frac{r \eta}{2} \cdot n ^ 2. \label{eq:density}
\end{align}
Using \eqref{eq:irregular}, \eqref{eq:inside} and \eqref{eq:density} gives at most $(r+1) \eta n^2$ edges of these three types. 

Clearly, the remaining edges of $G$ have endpoints in pairs that are regular for every color and must be assigned a color
that is dense with respect to the pair where its endpoints lie, i.e., their color must lie in the list of the corresponding edge of $\HH$. This means that the number of elements of $\mathcal{C}$ that can be associated with a given multicolored partition $V_1 \cup \cdots \cup V_m$ and a given $m$-vertex multicolored cluster graph $\HH$ is bounded above by 
\begin{align} \label{eq:multicolor_bound}
 \binom{n^2}{(r+1) \eta n^2} \cdot r^{(r+1) \eta n^2} \cdot \left( \prod_{j=1}^{r} j^{e_j(\HH)} \right)^{\left( \frac{n}{m} \right)^2},
\end{align}
where $e_j(\HH)$ denotes the number of edges of $\HH$ whose lists have size equal to $j$.
Here, we assume that $m$ divides $n$ to avoid dealing with lower order terms that can be absorbed into the error term $r^{(r+1) \eta n^2}$.
There are at most $M^n$ partitions of $V$ on $m \leq M$ classes and at most $2^{r\binom{m}{2}}$ multicolored cluster graphs with vertex set $[m]$.  Moreover, it is well-known that the entropy function satisfies
$$\binom{n^2}{(r+1) \eta n^2} \leq 2^{H((r+1)\eta n^2)}.$$

Thus, summing the upper bound~\eqref{eq:multicolor_bound} over all partitions and all corresponding multicolored cluster graphs, the number of $K_3^{(2)}$-free edge colorings of $G$ is at most
\begin{align}
M^n \cdot 2^{r M^2/2} \cdot 2 ^ {H((r+1) \eta) n^2} \cdot r^{(r+1) \eta n^2} \cdot \max_{\HH} \left( \prod_{j=1}^{r} j ^ {\frac{e_j(\HH)}{|V(\HH)|^2}} \right)^{n^2}. \label{eq:coloring_of_g}
\end{align}

Our aim is to find an upper bound on \eqref{eq:coloring_of_g}. The term $j=1$ in the product in \eqref{eq:coloring_of_g} does not affect the result. So, we define $\mathcal{S}=\mathcal{S}(G)$ to be the set of all subgraphs of multicolored cluster graphs $\HH$ of $G$ such that all edges are associated with lists of size at least two. Abusing the terminology, we also call the subgraph given by edges whose lists have size at least 2 the multicolored cluster graph associated with a coloring of $G$. Note that $\HH \in \mathcal{S}$ is $K_3^{(2)}$-free if and only if all lists associated with edges on a triangle are mutually disjoint. Given $\HH \in \mathcal{S}$, we let 
\begin{equation}\label{def_cH}
c(\HH)=\prod_{e\in E(\HH)}|L_e|^{\frac{1}{|V(\HH)|^2}}.
\end{equation}
We wish to find $\max_{\HH \in \mathcal{S} } c(\HH)$ to bound~\eqref{eq:coloring_of_g}.

As discussed in~\cite{2coloredlagos}, this is easy to do for the case of $r\leq 12$ colors (for completeness, we present the proof for $r\leq 12$ as an appendix). In the remainder of the proof, we focus on the remaining values of $r$ and consider the functions
\begin{eqnarray*}
r_0&=&r_0(r)=
\begin{cases} 
6& \textrm{ if } r=13\\
\lfloor r-2\sqrt{r}\rfloor& \textrm{ if } 14\leq r \leq 26.
\end{cases} \\%\label{def:r0}\\
r_1&=&r_1(r)=r_0+1.  %\label{def:r1}
\end{eqnarray*}
For $r\geq 13$, the crucial property in the definition of $r_1$ is that 
\begin{equation}\label{eq:Ar}
A(r)=\left\lfloor (r-r_1)/2 \right\rfloor \cdot \left\lceil (r-r_1)/2\right\rceil < r.
\end{equation}
Note that,  $A(r)\leq (r-r_1)^2/4$ and, as both factors of $A(r)$ are integers, we have $A(r) \leq r-1$.
The validity of~\eqref{eq:Ar}  may be verified directly for $r=13$, and if $r \geq 14$ we have $r-r_1=r-\lfloor r-2\sqrt{r}\rfloor-1<r-(r-2\sqrt{r}-1)-1=2\sqrt{r}$. 

Another important property of $r_1$ is that $2r_1 \geq r$. This may be verified directly for $r\in\{13,14,15\}$. For $r\geq 16$, we have 
$$2r_1 \geq 2(r-2\sqrt{r}) \geq r,$$
where the last inequality is equivalent to $\sqrt{r}\geq 4$. 

Our conclusion is a consequence of the following. 
\begin{claim}\label{claim13}
There is a multicolored cluster graph $\HH$ such that
\begin{align*}
e_{r_1}(\HH) + \dots + e_{r}(\HH) \geq \ex(m, K_3) - \xi m ^ 2.
\end{align*}
\end{claim}

Before addressing the proof of Claim~\ref{claim13}, we show that it implies the desired result. Let $\HH$ be an $m$-vertex multicolored cluster graph such that  $e_{r_1}(\HH) + \dots + e_{r}(\HH) \geq \ex(m, K_3) - \xi m ^ 2$. Let $\HH^{\prime}$ be the subgraph of $\HH$ with edge set $E_{r_1} (\HH)\cup \dots \cup E_{r}(\HH)$, where $E_j(\HH)$ denotes the set of edges of $\HH$ whose lists have size $j$. We claim that $\HH^{\prime}$ is $K_3$-free. Indeed, if  $\HH^{\prime}$ contained a triangle with edges $e_1,e_2,e_3$, then 
$$|L_{e_1}|+|L_{e_2}|+|L_{e_3}| \geq 3r_1> 2r_1+1>r$$ 
implies that two of the lists have non-empty intersection, leading to a $K_3^{(2)}$ in $G$ by Lemma~\ref{lemma:colored_subgraph}. 

By Theorem~\ref{theorem:stability_furedi}, there is a partition $U_1 \cup U_{2} = [m]$ with
\begin{align*}
e_{\HH^{\prime}}(U_1)  + \ e_{\HH^{\prime}}(U_{2}) \leq \xi m^2,
\end{align*}
where $e_{\HH^{\prime}}(U_i)$ is the number of edges of $\HH^{\prime}$ with both endpoints in $U_i$.  The bipartite subgraph $\widehat{\HH}$ obtained from $\HH^{\prime}$ obtained by removing all edges with both endpoints in the same class satisfies
\begin{align*}
e(\widehat{\HH}) \geq  (\ex(m, K_3) - \xi m ^ 2) - \xi m^2=  \ex(m, K_3) - 2 \xi m^2.
\end{align*}

We claim that, even if we add arbitrary edges with lists of size $1$ to $\HH$ (while preserving its property of being $K_3^{(2)}$-free),  $e_1(\HH) + \cdots + e_{r_0}(\HH)  \leq 10\xi m^2$. Otherwise, by our choice of $\xi$in~\eqref{eq:quantification}, Lemma~\ref{lemma:AY} can be applied and the graph obtained by adding the edges in $E_1 \cup \cdots \cup E_{r_0}$ to $\widehat{\HH}$ would contain a $K_3$ such that exactly one of the edges, say $f_1$, is in some set $U_i$. Let $f_2,  f_{3}$ be the other edges of the copy of $K_3$, which lie in $E_{r_1} \cup \cdots \cup E_{r}$. By construction, we have 
$$|L_{f_1}|+|L_{f_2}|+|L_{f_3}| \geq 1+2r_1>r,$$
a contradiction.

As a consequence, the number of edges of $\HH$ with both ends in the same set $U_i$ is at most $11 \xi m^2$. Let $W_i = \cup_{j \in U_i} V_j$ for $i \in \{1, 2\}$. Then, by our choice of $\eta$ and $\xi$ in~\eqref{eq:quantification}, we have
\begin{align*}
e_G(W_1) +  e_G(W_{2}) \leq r \eta n^2 + (n/m)^2 \cdot (e_\HH(U_1) + e_\HH(U_{2})) < \delta n^2,
\end{align*} 
as required. This proves Lemma~\ref{lemma:main_result}.

We now move to the actual proof of Claim~\ref{claim13}.  Given a cluster graph $\HH$, let $E_b(\HH)$ be the set of all edges whose color lists have sizes between $r_1$ and $r$. We refer to them as the \emph{blue} edges of $\HH$. Let $E_g(\HH)$ be the set of all edges whose color lists have sizes between $2$ and $r_0$, the \emph{green} edges of $\HH$.  The main ingredient in the proof of Claim~\ref{claim13} is the following auxiliary lemma.
\begin{lemma}\label{lemma:claim13}
Let $r$ be an integer such that $2\leq r \leq 26$ and let $\HH$ be a $(K_3,\leq 2)$-free multicolored cluster graph for which all edges are green. Then, for all $0 < \alpha \leq \frac{1}{1000}$ it is
$$c(\HH)\leq r^{\frac14-\alpha} <r^{\frac14}.$$
\end{lemma}

Before proving this lemma, we show that it implies the validity of Claim~\ref{claim13}. To this end, assume for a contradiction that any coloring of $G$ avoiding $K_{3}^{(2)}$ leads to a multicolored cluster graph $\HH$, where $|V(\HH)|=m$, for which
\begin{align}\label{eq:contradiction13}
e_{r_1}(\HH) + \dots + e_{r}(\HH) < \ex(m, K_3) - \xi m ^ 2. 
\end{align}

Let $\HH_0 = \HH$. If $\HH_0$ contains a blue edge, let $e_1 = \{ u, v \}$ be one such edge, so that $|L_{e_1}|\geq r_1$. Let $w$ be a vertex $V\setminus\{u,v\}$. If at most one of the pairs $\{u,w\}$ and $\{v,w\}$ is an edge of $\HH_0$, the list of colors on this edge cannot produce a $K_3^{(2)}$ involving $e_1$. If $w$ is a joint neighbor of $u$ and $v$, the sum of the sizes of $L_{\{u, z\}}$  and  $L_{\{v, z\}}$  is at most $r-|L_{e_1}| \leq r-r_1$, otherwise we obtain a copy of $K_3^{(2)}$. Thus the product of the sizes of $L_{\{u, z\}}$  and  $L_{\{v, z\}}$ is at most $A(r)=\left\lfloor (r-r_1)/2 \right\rfloor \cdot \left\lceil (r-r_1)/2\right\rceil < r$ by~\eqref{eq:Ar}. We say that a vertex $w$ has \emph{type 1} if there is a single edge connecting it to $\{u,v\}$, and this edge is blue. Otherwise it is said to have \emph{type 2}. If $w$ is a type 2 vertex, then either the list of the edge connecting it to $u$ or $v$ has size at most $r_0$ (in case $w$ is adjacent to at most one vertex among $u$ and $v$), or the product of the sizes of the lists is at most $A(r)$ (in case $w$ is adjacent to $u$ and $v$). This means that the product of the sizes of the lists of edges connecting a type 2 vertex to $\{u,v\}$ is at most $B(r)=\max\{r_0,A(r)\}<r$.    We know that $B(r) \leq r-1$. Note that
\begin{equation}\label{bound:B}
\frac{B(r)}{r} \leq \frac{r-1}{r} \leq r^{-\alpha}
\end{equation}
holds if $e^{-1/r} \leq  r^{-\alpha}$, which is equivalent to $\alpha \ln r - 1/r \leq 0$. The derivative of $f(r)=\alpha \ln r - 1/r$ is positive, which means that~\eqref{bound:B} is satisfied for all $r\leq 26$ if $\alpha \ln 26 - 1/26 \leq 0$, which holds for $\alpha \leq 0.0118$. 

Let $n_1(e_1)$ and $n_2(e_1)$ denote the number of vertices of type 1 and 2 with respect to $e_1$. Now remove vertices $u$  and $v$ and all incident edges from $\HH_0$, and call the remaining multicolored graph on $(m-2)$ vertices $\HH_1$. If $\HH_1$ contains a blue edge $e_2$, we repeat this argument for $\HH_1$ and for subsequent graphs until we reach a graph $\HH_{k_1}$ on $(m-2k_1)$ vertices that does not contain any blue edge, that is, such that every edge in $\HH_{k_1}$ is green. By construction, we have 
\begin{equation}\label{eq_number}
k_1+\sum_{i=1}^{k_1} (n_1(e_i)+n_2(e_i)) =  \sum_{i=1}^{k_1}  (m-2i+1) = k_1m -k_1^2.
\end{equation}
We conclude that the number of $K_3^{(2)}$-free colorings of $\HH$ is at most
\begin{eqnarray*}
 && r^{k_1+n_1(e_1)+\cdots+n_1(e_{k_1})} \cdot (B(r))^{n_2(e_1)+\cdots+n_2(e_{k_1})} \cdot c(\HH_{k_1})^{(m-2k_1)^2}\label{eq:26colors-1}   \\
 &=& \left(\frac{B(r)}{r}\right)^{n_2(e_1)+\cdots+n_2(e_{k_1})} \cdot  r^{k_1m -k_1^2} \cdot c(\HH_{k_1})^{(m-2k_1)^2}. \nonumber
\end{eqnarray*}

As all edges in the graph $\HH_{k_1}$ are green, using  the upper bound given in Lemma~\ref{lemma:claim13}, we obtain 
\begin{eqnarray} \label{eq:26colors1}
c(\HH)^{m^2} &\leq &   \left(\frac{B(r)}{r}\right)^{n_2(e_1)+\cdots+n_2(e_{k_1})} \cdot  r^{k_1m -k_1^2} \cdot r^{\left(\frac{1}{4}-\alpha\right)(m-2k_1)^2}\nonumber \\
&=& \left(\frac{B(r)}{r}\right)^{n_2(e_1)+\cdots+n_2(e_{k_1})}\cdot r^{\frac{m^2}{4}-\alpha(m-2k_1)^2}.
\end{eqnarray}
Recall that, by~(\ref{eq:contradiction13}), we are assuming that the number of blue edges of $\HH$ is at most $\ex(m, K_3) - \xi m ^ 2$. We shall consider two scenarios. First suppose that $k_1m -k_1^2 \leq \ex(m, K_3) - \xi m ^ 2$, so that $k_1 \leq m/2 - \sqrt{\xi}m$. In this case, all of the vertices in the above construction may have type 1 and~\eqref{eq:26colors1} is at most 
\begin{eqnarray*}
r^{\frac{m^2}{4}-\alpha \xi m^2} \leq r^{\frac{m^2}{4}-\frac{1}{10^4}\xi m^2}.
\end{eqnarray*}

In the second scenario, assume that $k_1 = m/2 - \sqrt{\xi}m+q$, for $q>0$ and $q \leq \sqrt{\xi}m$, so that, by~\eqref{eq_number} and our restriction on the number of blue edges, i.e., $k_1+\sum_{j=1}^{k_1} n_1(e_{j}) \leq \ex(m, K_3) - \xi m ^ 2$, we have
\begin{eqnarray*}
\sum_{j=1}^{k_1} n_2(e_{j})&=& k_1m-k_1^2-\left(k_1+\sum_{j=1}^{k_1} n_1(e_{j})\right)\\
&\geq& q(2\sqrt{\xi}m-q).
\end{eqnarray*}
Equation \eqref{eq:26colors1} is at most
\begin{eqnarray}\label{eq:26colors2} 
&& \left(\frac{B(r)}{r}\right)^{q(2\sqrt{\xi}m-q)} \cdot r^{\frac{m^2}{4}-\alpha(2\sqrt{\xi}m-2q)^2}.
\end{eqnarray}
If $q \leq 3\sqrt{\xi} m/4$, equation~\eqref{eq:26colors2} is at most
$$ r^{\frac{m^2}{4}-\alpha(2\sqrt{\xi}m-2q)^2} \leq r^{\frac{m^2}{4}- \frac{\alpha\xi m^2}{4}} \leq  r^{\frac{m^2}{4}-\frac{1}{10^4}\xi m^2}$$
for  $\alpha \geq 1/10^3$.

If $q \geq 3\sqrt{\xi} m/4$, equation~\eqref{eq:26colors2} is at most
$$\left(\frac{B(r)}{r}\right)^{\frac{15\xi m^2}{16}} \cdot r^{\frac{m^2}{4}} \stackrel{\eqref{bound:B}}{\leq} r^{\frac{m^2}{4}-\frac{1}{10^4}\xi m^2},$$
as for $2 \leq r \leq 26$
$$
\left(\frac{B(r)}{r}\right)^{\frac{15}{16}}  \cdot r^{\frac{1}{10^4}} \leq \left(\frac{r-1}{r}\right)^{\frac{15}{16}}  \cdot r^{\frac{1}{10^4}} \leq  e^{-\frac{15}{16r} + \frac{1}{10^4} \ln r } \leq 1.
$$

Combining the above cases, and using the upper bound~\eqref{eq:coloring_of_g}, we conclude that the number of $K_3^{(2)}$-free colorings of the graph $G$ satisfies
\begin{eqnarray*}
|\C_{r,(K_3,\leq 2)}(G)| &\leq &  M^n \cdot  2^{(H((r+1) \eta)) n^2 + r M^2 / 2} \cdot r^{(r+1) \eta n^2} \cdot\left(  r^{\frac{m^2}{4}-\frac{1}{10^4}\xi m^2}\right)^{\left( \frac{n}{m} \right) ^ 2} \\
 &\stackrel{n \gg 1}{\ll}& r^{\ex(n, K_3)},  \label{eq:result014}
\end{eqnarray*}
as $\xi > (10^4+1) \cdot (r+1) \cdot \eta + 10^4 \cdot H((r+1)\eta)$, which is a contradiction to the hypothesis that $ |\C_{r,(K_3,\leq 2)}(G)|  \geq r^{\ex(n,K_3)}$ and proves Claim~\ref{claim13}. 

To conclude the proof of Claim~\ref{claim13} (and thus of Lemma~\ref{lemma:main_result}), we still need to prove Lemma~\ref{lemma:claim13}. To this end, fix an integer $r$ such that  $13 \leq r \leq 26$, and let $\HH$ be a $K^{(2)}_3$-free multicolored cluster graph for which all edges are green. Recall that we are assuming that no edges of $\HH$ have lists of size less than two. 

\subsection{Proof of Lemma~\ref{lemma:claim13}}

The proof of Lemma~\ref{lemma:claim13} will be by induction, and we shall split the set $\mathcal{S}$ of $K_3^{(2)}$-free cluster graphs for which all edges are green into two classes. One such class, called $\mathcal{S}_1$ contains all $K_4$-free $\HH\in \mathcal{S}$ such that there is no copy of $K_3$ whose three edges have color lists of size at least four.%\footnote{We should mention that this result suffices to prove Lemma~\ref{lemma:claim13} if $r\leq 11$, as any triangle with lists of size at least $4$ would contain $K_3^{(2)}$. This can be used to prove Lemma~\ref{lemma:main_result} for $r\leq 11$.}  
\begin{lemma}\label{lemma26:s1<27}
Given $\HH \in \mathcal{S}_1$ with $m$ vertices, the following holds for $13 \leq r \leq 26$ and $0 < \alpha \leq \frac{1}{1000}$:
\begin{eqnarray*}
c(\HH)  \leq r^{\frac{1}{4} - \alpha}. 
\end{eqnarray*}
\end{lemma}

\begin{proof}
Fix $\HH \in \mathcal{S}_1$ with $m$ vertices. If $m=1$, there is nothing to prove. If $m=2$ and $r=13$, we have $c(\HH) \leq r_0 =6 \leq 13^{1-\alpha}$ for $\alpha \leq  0.3014 \leq (\log{13}-\log{6})/\log{13}$. For $r\geq 14$, we have 
$$c(\HH) \leq r_0 \leq r-2\sqrt{r} = r\left(1-\frac{2}{\sqrt{r}} \right) \leq r e^{-2/\sqrt{r}} \leq r^{1-\alpha}$$ for 
$\alpha \leq 0.1203 <\frac{2}{\sqrt{26} \ln 26} \leq \frac{2}{\sqrt{r} \ln r}$.
 
 Next assume $m\geq 3$. By the definition of $\mathcal{S}_1$,  the set $E_4\cup \cdots \cup E_{r_0}$ cannot induce a copy of $K_3$. Thus, by Tur\'{a}n's Theorem, we have
\begin{equation*}
|E_4| +\cdots + |E_{r_0}| \leq \ex(m, K_3) \leq \frac{1}{4}  m^2.
\end{equation*}

Again by the definition of $\mathcal{S}_1$,  the set $E_2\cup \cdots \cup E_{r_0}$ cannot contain a copy of $K_4$, leading to
\begin{equation*}
|E_2| +\cdots + |E_{r_0}| \leq \ex(m, K_4) \leq \frac{1}{3}  m^2.
\end{equation*}

If we set $x_i=|E_i|/m^2$ for $i \in \{2,\ldots,r_0\}$, by~\eqref{def_cH}, we have
$$c(\HH) = \prod_{j=2}^{r_0} j^{x_j},$$ so that the logarithm $\ln c(\HH)$ is bounded above by the solution to the linear program
\begin{eqnarray*}
&\max& \sum_{j=2}^{r_0} x_j \ln j\\
&\textrm{s.t.}&   \sum_{j=4}^{r_0} x_j\leq \frac{1}{4},~~
\sum_{j=2}^{r_0} x_j \leq \frac{1}{3}\\
&&x_j \geq 0,~j\in \{2,\ldots,r_0\}.
\end{eqnarray*}
Solving this linear program gives the optimal solution $x_{r_0}=1/4$, $x_3=1/12$ and $x_j=0$ for the remaining values of $j$. For $14\leq r\leq 26$, we have
\begin{eqnarray}
&& r_0^{\frac14} \cdot 3^{\frac{1}{12}} \leq \left( r - 2 \sqrt{r}\right)^{\frac{1}{4}}\cdot 3^{\frac{1}{12}} \leq r^{\frac{1}{4} - \alpha}\nonumber \\
&\Longrightarrow&  r^{\alpha} \cdot \left( 1 - \frac{2}{\sqrt{r}} \right)^{\frac{1}{4}} \cdot 3^{\frac{1}{12}} \leq  26^{\alpha} \cdot \left( 1 - \frac{2}{\sqrt{26}} \right)^{\frac{1}{4}} \cdot 3^{\frac{1}{12}} \leq  1. \label{eq_final}
\end{eqnarray}
For $r=13$, we get
\begin{equation}\label{eq_final2}
13^{\alpha} \cdot \left( \frac{6}{13} \right)^{\frac{1}{4}} \cdot 3^{\frac{1}{12}} \leq 1.
\end{equation}
 Equations~\eqref{eq_final} and~\eqref{eq_final2} hold for $0 < \alpha \leq 0.0101$. This leads to $c(\HH) \leq  r^{1/4 -\alpha}$.
\end{proof}

To complete the proof of Claim~\ref{claim13}, we consider the cluster graphs that are not considered in Lemma~\ref{lemma26:s1<27}. 
%We split the set $\mathcal{S}$ of $K_3^{(2)}$-free cluster graphs into the following classes: $\mathcal{S}_2$ contains all $\HH$ with at least one copy of $K_6$; $\mathcal{S}_3$ contains all $K_6$-free cluster graphs $\HH$ with at least one copy of $K_5$; $\mathcal{S}_4$ contains all $K_5$-free cluster graphs $\HH$ with at least one copy of $K_4$; $\mathcal{S}_5$ contains all remaining $\HH$ for which all edges are green, namely $K_4$-free graphs that contain a copy of $K_3$ whose three edges have color lists of size at least four.
In our arguments, we use the following optimization problem for given positive integers $p \geq 2$ and $L$:
\begin{eqnarray}\label{maxlemma}
&\max& \prod_{j=1}^c x_j \\
&\textrm{s.t.}& c, x_1,\dots , x_c\in {\mathbb N} = \{1,2,\ldots \} \nonumber \\
&& x_1+\cdots+x_c \leq p\nonumber\\
&& c \leq L.\nonumber
\end{eqnarray}

\begin{definition}\label{def:cs}
Given positive integers $k \geq 2$ and $r \geq 2$, let
\begin{itemize}
\item[(i)] $c_k(r)$ be the maximum of the optimization problem~\eqref{maxlemma} with $p=r\cdot \lfloor k/2 \rfloor$ and $L=\binom{k}{2}$; 

\item[(ii)] $c_k^*(r)$ be the maximum of the optimization problem~\eqref{maxlemma} with $p=r$ and $L=k$. 
\end{itemize}
\end{definition}

We shall use the following three straightforward lemmas.
 \begin{lemma}\label{gen_claim}
Let $k \geq 3$, let $\HH$ be a $K_3^{(2)}$-free multicolored cluster graph such that $|L_e| \geq 2$ for all $e \in E(\HH)$, and assume that $A \subset  V(\HH)$ is such that $\HH[A]$ is isomorphic to $K_k$. For a vertex $v \in V(\HH)$ let $E'(v)=\{\{v,x\} \in E(\HH) \colon x \in A\}$. For any $v \in V(\HH) \setminus A$, it holds that
 \begin{equation}\label{gen_UB}
 \prod_{e \in E'(v)} |L_e| \leq c^*_k(r) \leq \overline{c}_k(r)=\max\left\{\left(\frac{r}{j}\right)^{j} \colon j \in \{1,\ldots,k\}\right\}.
 \end{equation}
\end{lemma}

\begin{proof}
For each edge $e \in E'(v)$, set $x_e=|L_e|$. Because $A$ induces a clique and $\HH$ is $K_3^{(2)}$-free, the lists associated with edges between $v$ and $A$ are mutually disjoint, so that $\sum_{e \in E'(v)} x_e \leq r$. Let $j\leq k$ be the number of edges between $v$ and $A$. It is clear that 
$$\prod_{e \in E'(v)} |L_e| \leq \max\left\{\prod_{i=1}^j a_i \colon 1 \leq j \leq k, a_1,\ldots,a_j>0, a_1+\cdots+a_j \leq r\right\}=c_k^*(r).$$
The result follows because for $a_1+\cdots+a_j \leq r$ it is
$$\prod_{i=1}^j a_i\leq \left(\frac{r}{j}\right)^j.$$
\end{proof}

\begin{lemma}\label{lemma:UC}
Let $r$ and $k \geq 3$ be positive integers. For $j \geq 1$, consider a partition $E(K_{k}) = E_1 \cup \cdots \cup E_j$ of the edge set of the complete graph $K_k$ and integers $1 \leq s_1, \ldots , s_{j} \leq r$ such that
 $$r \left\lfloor \frac{k}{2} \right\rfloor < \sum_{i=1}^j |E_i| s_i.$$ 
Then, for any assignment of color lists in $[r]$ to the edges of $K_k$ such that, for each $i$, all edges $e \in E_i$ have list size at least $s_i$, there exists a copy of $K_3$ for which two of the lists have non-empty intersection.  
 \end{lemma}
 
\begin{proof}
Assume that  is an assignment of lists as in the statement such that, for all copies of $K_3$ in $K_k$, the lists associated with any two of its edges are disjoint. This means that, for every color $\alpha$, the edges whose lists contain $\alpha$ form a matching in $K_k$. Since a maximum matching in $K_k$ has size $\lfloor k/2 \rfloor$, we must have
$$ r   \left\lfloor \frac{k}{2} \right\rfloor \geq \sum_{e \in E(K_k)} |L_e| \geq  \sum_{i=1}^j |E_i| s_i,$$
contradicting our assumption about $r$ and $k$.
 \end{proof}

\begin{lemma}\label{lemma(ck)}
Let $r\geq 2$ and $k\geq 3$ be integers. Let $\HH$ be a $K_3^{(2)}$-free multicolored cluster graph whose underlying graph is $K_k$ and whose edge lists are contained in $[r]$ and have size at least two. Then 
$$\prod_{e \in E(\HH)} |L_e| \leq \tilde{c}_k(r)=\left(\frac{r}{\binom{k}{2}}\left\lfloor \frac k2 \right\rfloor \right)^{\binom{k}{2}}.$$
\end{lemma}

\begin{proof}
Given an edge $e \in E(\HH)$, let $x_e=|L_e|$. Let $E_i$ denote the set of edges of $\HH$ whose lists have size $i$. By Lemma~\ref{lemma:UC}, $\sum_{e \in E(\HH)} x_e = \sum_{i=2}^r i \cdot |E_i| \leq r \lfloor k/2 \rfloor$, since $\HH$ is $K_3^{(2)}$-free.

In particular, the vector $(x_e)_{e \in E(\HH)}$ is a feasible solution to the optimization problem~\eqref{maxlemma} with $p=r \lfloor k/2 \rfloor$ and $L=\binom{k}{2}$. For the inequality, observe that for any choice of $j$ positive real numbers such that $a_1+ \cdots + a_j \leq r\lfloor k/2\rfloor$, we have 
$$\prod_{i=1}^{\binom{k}{2}} a_i \leq \left(\frac{r}{j}\left\lfloor \frac k2 \right\rfloor \right)^{\binom{k}{2}}.$$
This concludes the proof.
\end{proof}

We are now ready to prove the desired result.
\begin{lemma}\label{lemma26:s2<27}
Fix an integer $r$ such that $13 \leq r \leq 26$. Given $\HH \in \mathcal{S} \setminus \mathcal{S}_1$ and $0 < \alpha\leq  \frac{1}{1000}$, we have
\begin{eqnarray*}\label{eq:lemma26}
c(\HH) \leq r^{\frac14 - \alpha}.
\end{eqnarray*}
\end{lemma}

\begin{proof}
Let $r \in \{13,\ldots,26\}$. For a contradiction, assume that the result is false and choose a counterexample $\HH \in \mathcal{S} \setminus \mathcal{S}_1$ with the minimum number of vertices. Recall that the edges of $\HH$ have lists with sizes between $2$ and $r_0$. Let $m$ be the number of vertices of $\HH$. We first show that $\HH$ is not isomorphic to a clique $K_m$ such that $3 \leq m \leq 6$.

For $\HH$ isomorphic to $K_3$, Lemma~\ref{lemma(ck)} tells us that $c(\HH)^{9} \leq \tilde{c}_3(r) = (r/3)^3$. Lemma~\ref{lemma26:s2<27} holds in this case because $c(\HH) \leq {((r/3)^3)}^{\frac{1}{9}} \leq r^{\frac{1}{4} - \alpha}$ for $r^{\frac{1}{12}+\alpha} \leq 26^{\frac{1}{12}+\alpha} \leq 3^{\frac13}$, which in turn holds for
$$\alpha \leq 0.02906<  \frac{4\ln{3}-\ln{26}}{12 \ln 26} .$$

For $\HH$ isomorphic to $K_4$, Lemma~\ref{lemma(ck)} gives $c(\HH)^{16} \leq \tilde{c}_4(r) =(r/3)^6$. Therefore, we need $c(\HH) \leq {((r/3)^6)}^{\frac{1}{16}} \leq  r^{\frac{1}{4} - \alpha}$, which holds for 
$$
\alpha \leq 0.00144< \frac{3\ln{3}-\ln{26}}{8 \ln 26}.
$$

If $\HH$ is isomorphic to $K_5$, Lemma~\ref{lemma(ck)} gives $c(\HH)^{25} \leq \tilde{c}_5(r) = (r/5)^{10}$. Therefore, $c(\HH) \leq ((r/5)^{10})^{\frac{1}{25}} ={(r/5)^{2/5}}  < r^{\frac{1}{4} - \alpha}$ if
$$
\alpha \leq 0.04759 < \frac{8 \ln 5 - 3 \ln 26}{20 \ln 26}.
$$

Finally, if $\HH$ is isomorphic to $K_6$, by  Lemma~\ref{lemma(ck)} the product of the sizes of the color lists of $\HH$ is at most  $\tilde{c}_6(r)$. We get 
\begin{eqnarray*}
c(\HH)^{36} &\leq& (\tilde{c}_6(r))^{\frac{1}{36}} = \left( \frac{r}{5} \right)^{\frac{15}{36}} 
\leq r^{\frac{1}{4} - \alpha},
\end{eqnarray*}
which holds for 
$$ \alpha \leq 0.03915<  \frac{15 \ln 5 - 6 \ln 26}{36 \ln 26}.
$$

Having established that $\HH$ is not isomorphic to a clique on $3\leq m \leq 6$ vertices, let $\omega(\HH)\geq 3$ denote the size of a maximum clique in $\HH$. If $\omega(\HH)\geq 6$, then the fact that $\HH\neq K_6$ implies that $m>6$. Fix $k=6$ and choose a set $A$ of vertices such that $A$ induces a copy of $K_6$ in $\HH$. Otherwise, let $k=\omega(\HH)$ and fix a set $A$ of vertices of size $k$ that induces a copy of $K_k$ in $\HH$. By the above, we know that $m>k$ in this case. Given a vertex $v \in V(\HH)\setminus A$, let $c_v$ be the product of the sizes of the lists on edges connecting $v$ to $A$. Clearly,
\begin{equation}\label{eq_UB}
c(\HH)^{m^2} = c(\HH[A])^{k^2} \cdot \left(\prod_{v \in V(\HH)\setminus A} c_v \right) \cdot  c(\HH[V(\HH)\setminus A])^{(m-k)^2}.
\end{equation}
We know that $c(\HH[A]),c(\HH[V(\HH)\setminus A]) \leq  r^{\frac{1}{4} - \alpha} $ by the minimality of $\HH$. 

If $k=6$, we have $c_v \leq \overline{c}_{6}(r)$ by Lemma~\ref{gen_claim}. Then~\eqref{eq_UB} leads to
\begin{eqnarray*}
c(\HH)^{m^2} &\leq &  r^{(\frac{1}{4}- \alpha)6^2} \cdot (\overline{c}_{6}(r))^{m-6} \cdot r^{(\frac{1}{4}- \alpha)(m-6)^2}\nonumber \\
&=& \left(\frac{\overline{c}_6(r)}{r^{3 - 12\alpha}}\right)^{m-6} \cdot  r^{(\frac{1}{4} - \alpha)m^2}.
\end{eqnarray*}
We conclude that $c(\HH)^{m^2} \leq  r^{(\frac{1}{4} - \alpha)m^2}$ because $\overline{c}_6(r)/r^{3-12\alpha} <1$ for $13 \leq r \leq 26$. Here, it suffices to verify that the quantity $\overline{c}_6(r)$ defined in~\eqref{gen_UB} satisfies $\overline{c}_6(r)<r^{3-12\alpha}$. We have 
$$\overline{c}_6(r)=
\begin{cases}
(r/6)^6& \textrm{ if } r\geq 15\\
(r/5)^5& \textrm{ if } 13\leq r\leq 14.\\
\end{cases}
$$ 
As it turns out, $\overline{c}_6(r)<r^{3-12\alpha}$ for 
$$\alpha \leq 0.0249< \frac{2\ln{6}-\ln{26}}{4 \ln{26}}.$$

For $k<6$, the maximality of the clique implies that any vertex  $v \in V(\HH)\setminus A$ has at most $k-1$ neighbors in $A$, so that~\eqref{eq_UB} becomes
\begin{eqnarray}\label{UB_otherk}
c(\HH)^{m^2} &\leq &  r^{(\frac{1}{4}- \alpha)k^2} \cdot (c^*_{k-1}(r))^{m-k} \cdot r^{(\frac{1}{4}- \alpha)(m-k)^2}\nonumber \\
&=& \left(\frac{c^*_{k-1}(r)}{r^{\frac{k}{2} - 2k\alpha}}\right)^{m-k} \cdot  r^{(\frac{1}{4} - \alpha)m^2}
\end{eqnarray}
 To obtain our result, we show that $c^*_{k-1}(r)<r^{\frac{k}{2} - 2k\alpha}$ for all suitable $k$ and $r$.
 
 If $k=5$, we use that $c^*_{4}(r)\leq \overline{c}_{4}(r)=(r/4)^4\stackrel{(i)}<r^{\frac{5}{2} - 10\alpha}$ for $2 \leq r\leq 26$. Indeed, (i) is equivalent to $r^{10\alpha+\frac{3}{2}}\leq 4^{4}$, which holds if 
 $$\alpha \leq 0.0201<\frac{8\ln{4}-3\ln{26}}{20 \ln{26}}.$$ 

If $k=4$, we use that $c^*_{3}(r)\leq \overline{c}_{3}(r)=(r/3)^3\stackrel{(ii)}<r^{2 - 8\alpha}$ for $2 \leq r\leq 26$. Indeed, (ii) holds for
$$\alpha \leq 0.0014 <\frac{3\ln{3}-\ln{26}}{8 \ln{26}}.$$ 
We observe that, for $r=27$, there is no $\alpha > 0$ for which~(ii) holds.

If $k<4$, we know by the hypothesis that $\HH \notin \mathcal{S}_1$ that $\mathcal{H}$ is $K_4$-free, but contains a copy of $K_3$ whose edges have lists of size at least four. We fix such a $3$-vertex set $A \subset V(\HH)$ that induces a copy of $K_3$ whose edges have lists of size at least four. If $v \in V(\HH) \setminus A$, then $v$ has at most two neighbors in $A$. If $v \in V(\HH) \setminus A$ has at most one neighbor in $A$, then its list has size at most $r_0$. If $v \in V(\HH) \setminus A$ has exactly two neighbors in $A$, say $v_1$ and $v_2$, then these edges form a triangle with an edge in the copy of $K_3$. Since the list of the edge $\{v_1, v_2 \}$ has size  at least four, the product $c_v$ of the sizes of the lists associated with the two edges between $A$ and $v$ is at most $(r-4)^2/4=s^*(r)$. We observe that $r_0 < s^*(r)$ for all $r\geq 13$. With this, the inequality~\eqref{UB_otherk} may be sharpened as
$$c(\HH)^{m^2} \leq  \left(\frac{(r-4)^2}{4r^{\frac{3}{2} - 6\alpha}}\right)^{m-k} \cdot  r^{(\frac{1}{4} - \alpha)m^2}.$$ 
Note that $(r-4)^2 \leq 4 \cdot r^{\frac{3}{2} - 6 \alpha}$ is equivalent to 
$$r^{\frac{1}{2}+6\alpha}\left(1-\frac{4}{r}\right)^2\leq 4.$$ For $\alpha>0$, the left-hand side is increasing as a function of $r$, so this holds for $13 \leq r \leq 26$ because $$26^{\frac{1}{2}+6\alpha}\left(1-\frac{4}{26}\right)^2 \leq 4$$ 
holds for $\alpha \leq 0.0046$. This concludes the induction step and proves Lemma~\ref{lemma26:s2<27}.
\end{proof}

\section{Final remarks and open problems}

The objective of this paper was to characterize the values of $r \geq 2$ for which the bipartite Tur\'{a}n graph $T_2(n)$ is the unique $(r,K_3^{(2)})$-extremal graph for all sufficiently large $n$. With Theorem~\ref{theorema:main}, we established that this holds precisely for $2\leq r\leq 26$. In this section, our aim is to put this result in a more general perspective and to discuss a few open problems.

Let $P$ be a pattern of a complete graph $K_k$. The following facts are known to hold (see~\cite{eurocomb15}):
\begin{itemize}

\item[(a)] If $P=K_k^R$, the rainbow pattern of $K_k$, then there exists $r_0$ such that the following holds for all $r\geq r_0$. There exists $n_0$ such that, for all $n \geq n_0$, the unique $n$-vertex $(r,P)$-extremal graph is the Tur\'an graph $T_{k-1}(n)$.

\item[(b)] If $P\neq K_k^R$, then there exists $r_1$ such that the following holds for all $r\geq r_1$. There exists $n_0$ such that, for all $n \geq n_0$, the Tur\'{a}n graph $T_{k-1}(n)$ is \emph{not} $(r,P)$-extremal.
\end{itemize}
This raises natural questions.
\begin{prob}
Given $k\geq 3$, let $r_0(k)$ be the least value of $r_0$ such that (a) holds. Determine $r_0(k)$.
\end{prob}
It is known that $r_0(3)=4$~\cite{baloghli}. For $k\geq 4$, upper and lower bounds on $r_0(k)$ have been provided in~\cite{multipattern,eurocomb15}, but we believe that the upper bounds in these papers are much larger than the actual value of this parameter. 

\begin{prob}
Given $k\geq 3$, characterize the $n$-vertex $(r,K_k^R)$-extremal graphs for $r<r_0(k)$. 
\end{prob}
Trivially, any such extremal graph must be isomorphic to $K_n$ for $r<\binom{k}{2}$, but it is not clear what happens for the remaining values of $r<r_0(k)$. In the case $k=3$, the only remaining value is $r=3$, and it is known that $K_n$ is the unique extremal graph in this case (for $n$ large.) To the best of our knowledge, no pair $(k,r)$ has been found for which there are arbitrarily large $(r,K_k^R)$-extremal graphs that are neither isomorphic to $T_{k-1}(n)$ nor to $K_n$.

\begin{prob}
Given $k\geq 3$ and a pattern $P$ of $K_k$ such that $P\neq K_k^R$, let $r_1(P)$ be the greatest value of $r_1$ such that (b) holds. Determine $r_1(P)$ and 
$$r_1(k)=\max \{r_1(P) \colon  P \textrm{ pattern of }K_k,~P\neq K_k^R\}.$$
\end{prob}
In this paper, we showed that $r_1(K_3^{(2)})=2$. It is known that, for any monochromatic pattern $K_k^M$, we have $r_1(K_k^M)=4$. In particular $r_1(3)=27$.

\begin{prob}\label{prob4}
Given $k\geq 3$ and a pattern $P$ of $K_k$ such that $P\neq K_k^R$, characterize the $n$-vertex $(r,P)$-extremal graphs for $r>r_1(P)$. 
\end{prob}
%It is easy to see that, for any $P$, there will be an increasing sequence $r_1<r_2<\cdots$ with the property that, for $i<j$ and sufficiently large $n$, any $n$-vertex $(r_j,P)$-extremal graph will be denser than any $n$-vertex $(r_i,P)$-extremal graph. 
For $P=K_3^{(2)}$ and $r=27$, we believe that $T_4(n)$ is the unique $(r,P)$-extremal graph for all sufficiently large $n$, but we do not have a proof of this. Moreover, the work in~\cite{BHS17} implies that, except for the patterns $K_k^M$ and $K_3^{(2)}$, the $(r,P)$-extremal graphs mentioned in Problem~\ref{prob4} must be complete multipartite graphs. However, it is not known whether the partition must always be equitable. Recent work by Botler et al.~\cite{botler} and by Pikhurko and Staden~\cite{PS2022} shows that, for some monochromatic patterns of complete graphs, unbalanced complete multipartite graphs can be very close to extremal, and may even be extremal in some cases.

\appendix

%\color{cyan}
\section{Proof of Lemma~\ref{lemma:main_result} for $r \leq 12$} 

As mentioned in the proof of  Lemma~\ref{lemma:main_result} for $r \geq 13$, the case of $r\leq 12$ is easier and will be presented here for completeness. It is based on the conference paper~\cite{2coloredlagos}. Indeed, one can alternatively prove this with the arguments in the case $r \geq 13$, but this would require more case analysis. The proof is identical to the case $r\geq 13$ up to~\eqref{eq:coloring_of_g}, and we again wish to find an upper bound on
\begin{align}
M^n \cdot 2^{\frac{rM^2}{2}} \cdot 2 ^ {H((r+1) \eta) n^2} \cdot r^{(r+1) \eta n^2} \cdot \max_{\HH} \left( \prod_{j=1}^{r} j ^ {\frac{e_j(\HH)}{|V(\HH)|^2}} \right)^{n^2}. \label{eq:coloring_of_g2}
\end{align}

Note that, for any fixed $m$-vertex multicolored cluster graph $\HH$, the inequality
\begin{eqnarray}\label{aux_eq1}
e_{\lfloor \frac{r}{3} \rfloor +1}(\HH)+\cdots+e_r(\HH) \leq \ex(m,K_3)
\end{eqnarray}
holds, as otherwise $\HH$ would contain a triangle such that the sum of the sizes of the lists of its edges is at least $3 \lfloor r/3 \rfloor +3>r$. By the pigeonhole principle, this gives a copy of $K_3^{(2)}$ in $\HH$ and therefore in the original colored graph by  Lemma~\ref{lemma:colored_subgraph}, a contradiction.

Equation~\eqref{aux_eq1} settles the problem for $2 \leq r \leq 5$. It implies that 
 $$e_2(\HH) + \cdots +e_r(\HH) \leq \ex(m,K_3).$$
 For $2 \leq r \leq 4$,
as in the proof of the case $r\geq 13$, see Claim~\ref{claim13},  if there exists $\HH$ such that 
$$e_2(\HH) + \cdots +e_r(\HH) \geq \ex(m,K_3)-\xi m^2,$$
we are done. Otherwise,  equation~\eqref{eq:coloring_of_g2} is at most
\begin{eqnarray*}
&& M^n \cdot 2^{\frac{rM^2}{2}} \cdot 2 ^ {H((r+1) \eta) n^2} \cdot r^{(r+1) \eta n^2} \cdot r^{\left(\ex(m,K_3) - \xi m^2) \right))\left(\frac{n}{m} \right)^2} \stackrel{n \gg 1}{<} r^{\ex(n, K_3)}
 \end{eqnarray*}
colorings, which contradicts our choice of $G$ in the statement of Lemma~\ref{lemma:main_result}. For $r=5$,  if there exists $\HH$ such that 
$$e_3(\HH) + e_4(\HH) +e_5(\HH) \geq \ex(m,K_3)-\xi m^2,$$
we are done again. Otherwise, if
$$e_3(\HH) + e_4(\HH) + e_5(\HH) \leq \ex(m,K_3)-\xi m^2,$$
keeping in mind that $e_2(\HH) + \cdots + e_5(\HH) \leq \ex(m,K_3)$, equation~\eqref{eq:coloring_of_g2} is at most
\begin{eqnarray*}
&& M^n \cdot 2^{\frac{5M^2}{2}} \cdot 2 ^{H(6 \eta) n^2} \cdot 5^{6 \eta n^2} \cdot \left( 2^{\xi m^2} \cdot 5^{\ex(m,K_3) - \xi m^2 }\right)^{\left(\frac{n}{m} \right)^2} \stackrel{n \gg 1}{<} r^{\ex(n, K_3)}
 \end{eqnarray*}
colorings, which contradicts again our choice of $G$ in the statement of Lemma~\ref{lemma:main_result}.

Assuming that $r \geq 6$, note that $r-3\geq \lfloor r/3 \rfloor +1$.  If there exists a multicolored cluster graph such that
$$
e_{r-3}(\HH)+\cdots+e_r(\HH) > \ex(m,K_3) -\xi m^2,
$$
we are again done using previous arguments, so we assume for a contradiction that this is not the case, that is, assume that
\begin{equation}\label{aux_eq2}
e_{r-3}(\HH)+\cdots+e_r(\HH) \leq \ex(m,K_3) -\xi m^2.
\end{equation} 

Given a $2$-element set $S \subset [r]$ and $j \in \{2,\ldots,r-4\}$, let $E_j(S,1;\HH)$ be the set of all edges  $e' \in E_j(\HH) $ that
satisfy  $|L_{e'} \cap S| \geq 1$, and let $e_j(S,1;\HH) =|E_j(S,;\HH)|$.
\begin{proposition} \label{prop:2}
Consider a multicolored cluster graph $\HH$ with no $K_3^{(2)}$. 
\begin{itemize}
\item[(a)] For each $2$-element subset $S \subseteq [r]$ of colors, the subgraph $\HH'$ of the multicolored cluster graph $\HH$ with edge set $\bigcup_{j=2}^{r-4} E_j(S,1 ;\HH) \cup \bigcup_{\ell = r-3}^r  E_{\ell}(\HH)$ is $K_3$-free. 
\item[(b)]  Moreover, there exists a $2$-element subset $S\subseteq [r]$ such that
\begin{eqnarray*} % \label{eq:gex1oo}
		&& \left|  \bigcup_{j=2}^{r-4} E_j(S,1;\HH) \right| 
		\geq  \sum_{j=2}^{r-4} \frac{  \binom{r}{2} - 
		 \binom{r-j}{2}   }{\binom{r}{2}} \cdot |E_j(\HH)|.
		\end{eqnarray*}
		\end{itemize}
\end{proposition}

\begin{proof}
We first argue that $  \bigcup_{j=2}^{r-4} E_j(S,1;\HH) \cup \bigcup_{\ell = r-3}^r E_{\ell}(\HH)$ is $K_3$-free. For a contradiction suppose that there is a triangle with edges $f_1,f_2,f_3$. Note that the lists of these edges have size at least two.

If one of the edges lies in $ \bigcup_{\ell = r-3}^r E_{\ell}(\HH)$, we may sum the  which produces a $K_3^{(2)}$, a contradiction. Next, assume that $f_1,f_2,f_3 \in  \bigcup_{j=2}^{r-4} E_j(S,1;\HH)$. Note that $|L_{f_1} \cap S|+|L_{f_2} \cap S|+|L_{f_3} \cap S| \geq 3$, so that two of the lists must contain the same color $a \in S$, and the third list contains a color $b \neq a$, which proves part~(a).

For part~(b), we claim that
$$ \sum_{S \in \binom{[r]}{2}} \sum_{j=2}^{r-4} |E_j(S,1;\HH)| =  \sum_{j=2}^{r-4} \left(  \binom{r}{2} - 
		\binom{r-j}{2} \right) \cdot e_j(\HH).$$
Indeed, for $j =2, \ldots ,  r-4$, every edge  $e \in E_j(\HH)$ is counted on the left hand side for all sets $S \in \binom{[r]}{2}$ 
such that $|S \cap e| \geq 1$, which amounts to
$ 
		 \binom{r}{2} - 
 \binom{r-j}{2} 
		$
times.

By averaging, as there are $\binom{r}{2}$ distinct $2$-element subsets in $[r]$,
	there exists a $2$-element subset $S \subseteq [r]$ such that 
	\begin{eqnarray*}
		&& \left|\bigcup_{j=2}^{r-4} E_j(S,1;\HH) \right| 
		\geq \sum_{j=2}^{r-4} \frac{   \binom{r}{2} - 
		\binom{r-j}{2}  }{\binom{r}{2}}
		  \cdot e_j(\HH).
		\end{eqnarray*}
\end{proof}

Proposition~\ref{prop:2}~(b) implies the following inequality:
\begin{eqnarray} \label{eq:lin_A2}
&& \sum_{j=2}^{r-4} \frac{  \binom{r}{2} - 
		 \binom{r-j}{2}   }{\binom{r}{2}} \cdot e_j(\HH) + \sum_{\ell = r-3}^r e_\ell(\HH) \leq \ex(m,K_3).
		 \end{eqnarray}

Going back to~\eqref{eq:coloring_of_g2}, we have
\begin{eqnarray}\label{eq:4}
M^n \cdot 2^{\frac{r M^2}{2}} \cdot 2 ^ {H((r+1) \eta) n^2} \cdot r^{(r+1) \eta n^2} \cdot \max_{\HH} \left( \prod_{j=1}^{r} j ^ {\frac{e_j(\HH)}{|V(\HH)|^2}} \right)^{n^2}.
\end{eqnarray}
Note that finding the maximum in this equation is equivalent to maximizing
\begin{equation*}
e_2(\HH) \ln{2}+ e_3(\HH) \ln{3} + \cdots + e_r(\HH) \ln{r},
\end{equation*}
which is a linear objective function with respect to the variables $e_2(\HH),\ldots,e_r(\HH) \geq 0$. Together with linear constraints in~(\ref{aux_eq2}) and (\ref{eq:lin_A2}), we obtain a linear program as follows. Given $\HH$, set $\zeta(\HH)=\left( \ex(m,K_3) -e_{r-3}(\HH)-\cdots-e_r(\HH) \right)/m^2$, so that $\zeta(\HH) \geq \xi$. The inequalities~(\ref{aux_eq2}) and~(\ref{eq:lin_A2}) tell us that to find an upper bound on~(\ref{eq:4}), we may consider the linear program
\begin{eqnarray}
&&\max ~~  x_2 \ln{2}+ x_3 \ln{3} + \cdots + x_{r-4} \ln{(r-4)}  \label{eq:linear} \\
&&\sum_{j=2}^{r-4} \frac{  \binom{r}{2} - \binom{r-j}{2}   }{\binom{r}{2}} \cdot x_j  \leq 1 \nonumber \\
&&x_2,\ldots, x_{r-4} \geq 0, \nonumber
\end{eqnarray}
where $x_i$ plays the role of $e_i(\HH)/(\zeta(\HH) m^2)$. As it turns out, for $r \in \{6,\ldots,12\}$, if $y(r)$ is the optimum of the linear program, the value of $Y(r)=e^{y(r)}$ is given in Table~\ref{table:1}.  

\begin{table}[h]
\begin{center}
\begin{tabular}{|c|c|c|c|c|c|c|c|}
\hline
   $r$ &  $6$ & $7$ & $8$ & $9$ &  $10$ & $11$ & $12$ \\
\hline
$Y(r)$ & $2^{5/3} \approx$ 3.17 & $3^{7/5} \approx$  4.65&  $4^{14/11} \approx 5.84$ & $5^{6/5} \approx 6.90$ & $4^{3/2}=8$ & $4^{55/34} \approx 9.42$ & $3^{11/5} \approx 11.21$  \\
\hline
\end{tabular}
\end{center}
\caption{Values of $Y(r)$.}
\label{table:1}
\end{table}

Clearly, for any multicolored cluster graph $\HH$, we have 
\begin{eqnarray}\label{eq_Mr}
\prod_{j=1}^r j^{e_j(\HH)} &=& \left( \prod_{j=1}^{r-4} j^{e_j(\HH)} \right) \cdot \left( \prod_{j=r-3}^{r}  j^{e_j(\HH)} \right) \nonumber\\
&\leq& \left( \prod_{j=1}^{r-4} j^{e_j(\HH)} \right) \cdot r^{e_{r-3}(\HH)+\cdots+e_r(\HH)} \nonumber\\
& \leq& Y(r)^{\zeta(\HH) m^2} \cdot r^{\ex(m,K_3) -\zeta(\HH) m^2} \leq Y(r)^{\xi m^2} \cdot r^{\ex(m,K_3)-\xi m^2}.
\end{eqnarray} 
Plugging this into~\eqref{eq:4}, as $Y(r)<r$, we see that $G$ has fewer than $r^{\ex(n,K_3)}$ colorings, the desired contradiction.

\end{document}